\documentclass{amsart}

\newtheorem{theorem}{Theorem}[section]
\newtheorem{lemma}[theorem]{Lemma}
\newtheorem{proposition}{Proposition}

\theoremstyle{definition}

\theoremstyle{remark}

\numberwithin{equation}{section}

\begin{document}

\title{Bottom Spectrum Estimate Under Curvature Integrability Condition}

\author{Cole Durham}
\address{Department of Mathematics, University of Connecticut, Storrs, CT 06268, USA}
\email{cole.durham@uconn.edu}


\begin{abstract}
In this paper we prove an upper bound for the bottom of the spectrum of the Laplacian on manifolds with Ricci curvature bounded in integral sense. Our arguments rely on the existence of a minimal positive Green's function and its properties.
\end{abstract}

\maketitle

\section{INTRODUCTION}
The goal of this work is to gain information about the bottom of the spectrum of the Laplacian on manifolds which have bounded integral Ricci curvature as opposed to the classical assumption $\text{Ric} \geq -(n-1)K$ for $K$ a real constant. The study of integral curvature bounds has been extensive over the past thirty years, and we mention several relevant results below. It is noteworthy that the spectrum of the Laplacian plays a role not just in pure geometry, but has also been linked to Brownian motion through the works of Varopoulos \cite{Var}, Grigor'yan \cite{AG}, etc., and to graph theory via Cheeger's inequality \cite{d63036efc9d24f07b8908864667e28aa}.

Recall that the bottom of the spectrum of the Laplacian on a complete manifold $M$ is defined as
\begin{equation*}
\lambda_0(M) = \inf_{\phi\in C_c^{\infty}(M)}\frac{\int_M|\nabla \phi|^2}{\int_M\phi^2}
\end{equation*}
where the infimum is taken over smooth functions $\phi$ with compact support. An early estimate of $\lambda_0(M)$ in the setting of pointwise Ricci curvature bounds was proven by Cheng \cite{cheng} in 1975.

\begin{theorem}[Cheng] Suppose $M$ is a complete manifold of dimension $n$ with the property $\text{Ric} \geq -(n-1)K$ for $K>0$. Then the bottom of the spectrum of the Laplacian satisfies
\begin{equation*}
\lambda_0(M)\leq \frac{(n-1)^2K}{4}.
\end{equation*}
\end{theorem}

Several steps have been taken to generalize this result by replacing the pointwise curvature bound with integral curvature bounds. Around the turn of the century, Petersen and Sprouse \cite{PS} utilized the function
\begin{equation*}
\rho(x) = \max\{0,(n-1)K-\text{Ric}_-(x)\}
\end{equation*}
defined for $x\in M$ and $K\in \mathbb{R}$, where $\text{Ric}_-(x)$ is the smallest eigenvalue of the Ricci tensor viewed as a map from $T_xM$ to itself. Moreover, they defined
\begin{equation*}
\overline{\kappa}(s,K,R)=\sup_{x\in M}\frac{1}{\text{vol}\,B_x(R)}\int_{B_x(R)}\rho^s.
\end{equation*}
The function $\overline{\kappa}$ generalizes the pointwise curvature assumption in a natural way, since $\overline{\kappa}(s,K,R)=0$ if and only if $\text{Ric}\geq (n-1)K$; when $K<0$ this is equivalent to the previous pointwise condition. By allowing $\overline{\kappa}$ to take on positive values, one permits the Ricci curvature of $M$ to fall below what had been a rigid lower bound. In particular, by assuming $\overline{\kappa}$ is sufficiently small and $s>\frac{n}{2}$ ($n$ being the dimension of the manifold), Petersen and Sprouse proved the following estimate.

\begin{theorem}[Petersen, Sprouse] Given $\delta>0$, there exists $\varepsilon = \varepsilon(n,s,K,R)>0$ such that every $n$-dimensional Riemannian manifold with $\overline{\kappa}(s,K,R)\leq \varepsilon$ satisfies
\begin{equation*}
\lambda_0^D(B_{x_0}(R))\leq (1+\delta)\lambda_0^D(n,R,K)
\end{equation*}
where $\lambda_0^D(B_{x_0}(R))$ is the first eigenvalue for the Dirichlet problem on the metric ball $B_{x_0}(R)$ centered at any $x_0\in M$, while $\lambda_0^D(n,R,K)$ is the analogous eigenvalue on a metric ball in the space form $S_K^n$. 
\end{theorem}
In the result of Petersen-Sprouse, if $R\to\infty$ and $\delta\to 0$, the result of Cheng is obtained for $\lambda_0(M)$.

More recently, in 2017 Seto and Wei \cite{SW} proved an estimate for the spectrum of the $p$-Laplacian, $p>1$, on a metric ball, extending the result of Petersen and Sprouse to a nonlinear differential operator. Using $\rho$ as in Theorem 1.2, their result may be stated as follows.

\begin{theorem}[Seto, Wei] Suppose $M$ is a complete Riemannian $n$-manifold. For any $x\in M$, $K\in\mathbb{R}$, $r>0$, $p>1$, $s>\frac{n}{2}$, and $\overline{s}=\max\left\{s,\frac{p}{2}\right\}$, there exists $\varepsilon = \varepsilon(n,p,\overline{s}, K,r)>0$ such that if
\begin{equation*}
\|\rho\|_{\overline{q},B_x(r)}^*=\left(\frac{1}{\text{\rm vol}\,B_x(r)}\int_{B_x(r)}\rho^{\overline{s}}\right)^\frac{1}{\overline{s}}<\varepsilon,
\end{equation*}
then
\begin{equation*}
\lambda_{0,p}(B_x(r))\leq \overline{\lambda}_{0,p}(B_K(r))+C\left(\|\rho\|_{\overline{q},B_x(r)}^*\right)^{\frac{1}{2}}
\end{equation*}
where $B_K(r)$ is a ball of radius $r$ in the space form $S_K^n$ and $\lambda_{0,p},\overline{\lambda}_{0,p}$ indicate the first Dirichlet eigenvalues of the $p$-Laplacian in $M$ and $S_K^n$ respectively. \end{theorem}

In broad terms, the latter two results assume that the amount of Ricci curvature lying below $(n-1)K$ is small in the sense of $L^s(B_x(r))$ for some $s>\frac{n}{2}$, all $x\in M$, and all $r>0$. Indeed, according to Gallot \cite{Gallot} it is impossible to construct interesting examples of manifolds with this curvature assumption and $s\leq \frac{n}{2}$. To allow for $s\leq\frac{n}{2}$, we assume a slightly different integral curvature bound and employ distinct techniques from those used in the aforementioned works. For notation, we will label
\begin{equation*}
w = \left\{\inf_{|v|=1}\{\text{Ric}(v,v)+(n-1)\}\right\}_{-}.
\end{equation*}
By taking the negative part as shown, we obtain a function which measures how much the Ricci curvature falls below $-(n-1)$ in the same fashion as Petersen-Sprouse and Seto-Wei. Under the assumption that $w$ satisfies an integrability condition on $M$, we show that the eigenvalue estimate of Cheng still holds.

\begin{theorem} Let $M$ be a complete $n$-dimensional manifold with $w\in L^s(M)$ for $s\geq \frac{3}{2}$. Then 
    \begin{equation*}
    \lambda_0(M)\leq \frac{(n-1)^2}{4}.
    \end{equation*}
\end{theorem}

Thus under a weaker curvature restriction than was used in the original work of Cheng we obtain the same spectral bound. It is not yet known if the hypothesis $s\geq \frac{3}{2}$ is sharp, though a result in this direction is proven in section 3. Our approach to the proof of Theorem 1.4 relies on the Green's function for the Laplacian on $M$. For $p\in M$ fixed, denote the Green's function based at $p$ by $G(x) = G(p,x)$, meaning $\Delta G(x) = -\delta_p(x)$. In particular, a manifold $M$ is called nonparabolic if it admits a strictly positive Green's function; this is guaranteed provided $\lambda_0(M)>0$, which is not restrictive in relation to Theorem 1.4. The arguments given here are inspired by \cite{MSW} and \cite{LW}, from which we utilize several facts.

\begin{proposition}[Li, Wang] Let $M$ be a complete, nonparabolic manifold of dimension $n$. If $G(x)$ is the minimal positive Green's function on $M$ and it has level sets $L(a,b) = \{x\in M\,:\,a<G(x)<b\}$, the following estimates hold:

 \begin{enumerate}
    \item  
    \begin{equation*}
    \int_{M\setminus B_p(R)}|\nabla G|^2\leq Ce^{-2\sqrt{\lambda_0(M)}R};
    \end{equation*}
    in particular
    \begin{equation*}
    \int_{M\setminus B_p(1)}|\nabla G|^2<\infty.
     \end{equation*}

    \item Similarly,
    \begin{equation*}
    \int_{M\setminus B_p(R)}G^2\leq Ce^{-2\sqrt{\lambda_0(M)}R}.
    \end{equation*}

    \item For any integrable function $f$,
    \begin{equation*}
    \int_{L(a,b)}|\nabla G|^2f(G)=\int_a^b f(t)dt.
    \end{equation*}
\end{enumerate} 
\end{proposition}
We now proceed toward the proof of Theorem 1.4.
\section{PROOF OF THEOREM 1.4}
Our first step is to prove a lemma which will be useful in later estimates.

\begin{lemma}
    Let $M$ be a complete $n$-dimensional manifold with $\lambda_0(M)>0$. Then the minimal positive Green's function $G$ with pole at $p\in M$ satisfies
\begin{equation*}
\int_{M\setminus B_p(1)} \frac{G}{\ln^{2\gamma}(1+G^{-1})}<\infty
\end{equation*}
for all $\gamma > \frac{1}{2}$.
\end{lemma}

\begin{proof}
Let $f(G) = \frac{1}{\ln^{\gamma}(AG^{-1})}$, where $A=e^\frac{1}{\alpha}\max_{\partial B_p(1)} G$, $\alpha>0$ small, and $\gamma>\frac{1}{2}$. We wish to estimate $\int_{M\setminus B_p(1)} Gf^2$. To do so, define $\varphi = \chi \psi$ based on the level sets of $G$ (as stated in Proposition 1) with
\[
\chi(x) = 
\begin{cases} 
    0               &\quad x\in L(0,\varepsilon^2)\\
    \frac{\ln G(x) - \ln(\varepsilon^2)}{-\ln(\varepsilon)}         &\quad x\in L(\varepsilon^2,\varepsilon)\\
    1               &\quad x\in L(\varepsilon,\infty)
\end{cases}
\]
and 
\[
\psi(x) = 
\begin{cases}
    0                 & \quad x \in B_p(1) \\
    r(x) - 1          & \quad x \in B_p(2) \setminus B_p(1) \\
    1                 & \quad x \in B_p(R) \setminus B_p(2) \\
    R+1 - r(x)        & \quad x \in B_p(R+1) \setminus B_p(R) \\
    0                 & \quad x \in M \setminus B_p(R+1)
\end{cases}
\]
where $0<\varepsilon<1$ and $r(x)$ is the geodesic distance from $p$ to $x\in M$. From the Poincar\'e inequality and the assumption $\lambda_0(M)>0$ we know that 
\begin{equation*}
\int_M (G^\frac{1}{2}f\varphi)^2\leq \frac{1}{\lambda_0(M)}\int_M \left|\nabla (G^\frac{1}{2}f\varphi)\right|^2.
\end{equation*}
We expand the gradient on the right hand side above. For any differentiable functions $g$ and $h$, the fact that $|\nabla (gh)|^2\leq 2|\nabla g|^2h^2+2g^2|\nabla h|^2$ implies
\begin{align*}
\left|\nabla\left(G^\frac{1}{2}f \varphi\right)\right|^2 & \leq 2|\nabla(G^\frac{1}{2}\ln^{-\gamma}(AG^{-1}))|^2\varphi^2+2Gf^2|\nabla\varphi|^2\\
&\leq 2\left(\frac{2|\nabla (G^\frac{1}{2})|^2}{\ln^{2\gamma}(AG^{-1})}+2G|\nabla(\ln^{-\gamma}(AG^{-1}))|^2\right)\varphi^2+2Gf^2|\nabla\varphi|^2.
\end{align*}
Since $2G|\nabla(\ln^{-\gamma}(AG^{-1}))|^2 = 2\gamma^2\frac{|\nabla G|^2}{G\ln^{2\gamma+2}(AG^{-1})}$ and $\ln(AG^{-1})\geq 1$ on $M\setminus B_p(1)$, we simplify and find

\begin{equation*}
\left|\nabla\left(G^\frac{1}{2}f\varphi\right)\right|^2\leq 2(1+2\gamma^2)\frac{|\nabla G|^2f^2\varphi^2}{G}+2Gf^2|\nabla\varphi|^2.
\end{equation*}
We may also expand $\nabla\varphi = \nabla(\chi\psi)$ in a similar fashion. Note that $\nabla\chi$ is nonzero only on $L(\varepsilon^2,\varepsilon)$ and in this region $|\nabla\chi|^2\leq \frac{1}{\ln^2(\varepsilon)}\frac{|\nabla G|^2}{G^2}$ by direct computation. Further conditions imposed from the definitions of $\chi$ and $\psi$ show
\begin{align}
\int_M \left| \nabla(G^{\frac{1}{2}} f \varphi) \right|^2 
&\leq 2(1 + 2\gamma^2) \int_M \frac{|\nabla G|^2 f^2 \varphi^2}{G} \notag \\
&\quad + \frac{4}{\ln^2(\varepsilon)} \int_{L(\varepsilon^2, \varepsilon) \cap (B_p(R+1)\setminus B_p(1))} \frac{|\nabla G|^2 f^2}{G}\notag  \\
&\quad + 4 \int_{(B_p(R+1)\setminus B_p(R)) \cup (B_p(2)\setminus B_p(1))} G f^2 \chi^2 \notag
\end{align}
The first two terms are integrable using the co-area formula. For the third term, first note that $\int_{B_p(2)\setminus B_p(1)} Gf^2\chi^2$ is bounded by a universal constant $C$ independent of $R$. Second, since $f^2\leq 1$ on $M\setminus B_p(1)$ and $\chi = 0$ outside of $L(\varepsilon^2,\infty)$ we have using Proposition 1
\begin{align*}
\int_{(B_p(R+1)\setminus B_p(R))\cup (B_p(2)\setminus B_p(1))} Gf^2\chi^2 
&\leq \frac{1}{\varepsilon^2}\int_{M\setminus B_p(R)} G^2 + C \notag \\
&\leq \frac{1}{\varepsilon^2}e^{-2\sqrt{\lambda_0(M)}R}+C.
\end{align*}
What we have shown is

\begin{align*}
\int_M Gf^2\varphi^2\leq \frac{1}{\lambda_0(M)}\left(2(1+2\gamma^2)C+\frac{C}{\ln^2(\varepsilon)}+\frac{1}{\varepsilon^2}e^{-2\sqrt{\lambda_0(M)}R}+C\right).
\end{align*}
Letting $R\to\infty$ and then $\varepsilon\to 0$ gives the inequality
\begin{equation*}
\int_{M\setminus B_p(2)} Gf^2\leq C(\gamma,\lambda_0(M))
\end{equation*}
which implies the desired result.
\end{proof}

We now prove a second integral estimate which will be needed shortly. The proof of this estimate relies on our curvature assumption where the proof of Lemma 2.1 did not.

\begin{lemma}
    Let $M$ be a complete $n$-dimensional manifold with $\lambda_0(M)>0$ and suppose $w\in L^s(M)$ for some $s\geq \frac{3}{2}$. Then the minimal positive Green's function $G$ with pole at $p\in M$ satisfies
    \begin{equation*}
    \int_{M\setminus B_p(1)} \frac{|\nabla G|^3}{G^2\ln^{2\gamma}(1+G^{-1})}<\infty
    \end{equation*}
    for all $\gamma > \frac{1}{2}$.
\end{lemma}

\begin{proof}

   Let $v = \ln(G)$. Our goal is to estimate $\int_M |\nabla v|^3\phi^2$ where $\phi$ is a suitably chosen cutoff function which vanishes on $B_p(1)$. A standard manipulation of the Bochner formula combined with the definition of $w$ gives 
\begin{align*}
\frac{1}{n-1}|\nabla v|^3\leq& \Delta |\nabla v|-\frac{1}{n-1}|\nabla |\nabla v||^2|\nabla v|^{-1}+\frac{n-2}{n-1}\left<\nabla |\nabla v|^2,\nabla v\right>|\nabla v|^{-1}\\
&+w|\nabla v|+(n-1)|\nabla v|.
\end{align*}
From this inequality, for any cutoff function $\phi$ we have
\begin{align*}
\frac{1}{n-1}\int_M |\nabla v|^3\phi^2 \leq &\int_M (\Delta |\nabla v|)\phi^2 - \frac{1}{n-1}\int_M |\nabla |\nabla v||^2|\nabla v|^{-1}\phi^2\\
&+\frac{n-2}{n-1}\int_M\left<\nabla |\nabla v|^2,\nabla v\right>|\nabla v|^{-1}\phi^2+\int_M w|\nabla v|\phi^2\\
&+(n-1)\int_M |\nabla v|\phi^2. 
\end{align*}
Using integration by parts, this is equivalent to
\begin{align*}
\frac{1}{n-1}\int_M |\nabla v|^3\phi^2\leq &-\int_M \left<\nabla \phi^2,\nabla |\nabla v|\right>- \frac{1}{n-1}\int_M |\nabla |\nabla v||^2|\nabla v|^{-1}\phi^2 \\
&+\frac{2(n-2)}{n-1}\int_M\left<\nabla |\nabla v|,\nabla v\right>\phi^2 +\int_M w|\nabla v|\phi^2\\
&+(n-1)\int_M|\nabla v|\phi^2.
\end{align*}

Now we choose $\phi = G^{\frac{1}{2}}\eta$ for some cutoff function $0\leq \eta\leq 1$ such that $\eta = 0$ on $B_p(1)\cup (M\setminus B_p(2R))$. Expanding the right side of the previous inequality, we arrive at
\begin{align*}
\frac{1}{n-1}\int_M |\nabla v|^3\phi^2\leq &- \int_M \left<\nabla G, \nabla|\nabla v|\right>\eta^2-2\int_M \left<\nabla \eta, \nabla |\nabla v|\right>G\eta \\
&-\frac{1}{n-1}\int_M |\nabla |\nabla v||^2|\nabla v|^{-1}G\eta^2 + (n-1)\int_M|\nabla v|G\eta^2\\
&+\int_M w|\nabla v| G\eta^2  +\frac{2(n-2)}{n-1}\int_M \left<\nabla G, \nabla|\nabla v|\right> \eta^2.
\end{align*}
Integration by parts and the fact that $G$ is harmonic away from $p$ show
\begin{align}\label{eq1}
\frac{1}{n-1}\int_M |\nabla v|^3\phi^2 \leq &\left(1-\frac{2(n-2)}{n-1}\right)\int_M \left<\nabla G,\nabla\eta^2\right>|\nabla v|-2\int_M\left<\nabla\eta,\nabla |\nabla v|\right>G\eta \notag \\
&-\frac{1}{n-1}\int_M|\nabla |\nabla v||^2|\nabla v|^{-1}G\eta^2+\int_M w|\nabla v| G\eta^2\notag\\
&+ (n-1)\int_M|\nabla v|G\eta^2.  
\end{align}
We apply Cauchy's inequality with $\delta>0$ to each of the first two terms in (\ref{eq1}). First,
\begin{align*}
\left(1-\frac{2(n-2)}{n-1}\right)\int_M \left<\nabla G,\nabla\eta^2\right>|\nabla v|&\leq C\int_M \eta |\nabla \eta||\nabla v|^2 G\\
&\leq C\delta \int_M |\nabla v|^3G\eta^2 + \frac{C}{\delta}\int_M |\nabla v||\nabla \eta|^2 G.
\end{align*}
By similar reasoning
\begin{align*}
-2\int_M\left<\nabla\eta,\nabla |\nabla v|\right>G\eta &\leq C\int_M |\nabla\eta||\nabla|\nabla v|| G\eta \\
&\leq C\delta \int_M |\nabla|\nabla v||^2|\nabla v|^{-1}G\eta^2 +\frac{C}{\delta}\int_M|\nabla\eta|^2G|\nabla v|.
\end{align*}
We replace terms in (\ref{eq1}) by their respective bounds. Under the assumption $\delta\leq \frac{C}{n-1}$, we simplify and have
\begin{align}
\frac{1}{n-1}\int_M|\nabla v|^3\phi^2 \leq& C\delta\int_M |\nabla v|^3G\eta^2 + \left(C\delta - \frac{1}{n-1}\right)\int_M|\nabla |\nabla v||^2|\nabla v|^{-1}G\eta^2 \notag \\
&+ \frac{C}{\delta}\int_M (\eta^2+|\nabla\eta|^2)G|\nabla v|+\int_Mw|\nabla v|G\eta^2.\notag
\end{align}
Assuming $\delta$ is small enough such that $\frac{1}{n-1} - C\delta > 0$, given $|\nabla|\nabla v||^2|\nabla v|^{-1} G\eta^2 \geq 0$ we find
\begin{equation}\label{eq2}\left(\frac{1}{n-1}-C\delta\right)\int_M |\nabla v|^3\phi^2 \leq \frac{C}{\delta}\int_M(\eta^2+|\nabla\eta|^2)G|\nabla v|+\int_M w|\nabla v|G\eta^2.\end{equation}

To continue toward achieving a bound, we specify $\eta = \tau f(G)$ where $0\leq \tau \leq 1$ is a cutoff function satisfying $\tau = 0$ on $B_p(1)\cup (M\setminus B_p(2R)),$ $\tau = 1$ on $B_p(R)\setminus B_p(2)$. As in Lemma 2.1,
\begin{equation*}
f(G) = \frac{1}{\ln^\gamma(AG^{-1})}, \;\; A = e^{\frac{1}{\alpha}} \max_{\partial B_p(1)} G,\;\; \gamma>1/2,
\end{equation*}
but here we require $0<\alpha<\delta$ is small relative to $\delta$. By Young's inequality for $0<\beta<\delta$ we know
\begin{align}\label{eq3}
\int_M |\nabla v|G\eta ^2&\leq \frac{\beta}{3}\int_M\frac{|\nabla G|^3}{G^2}f^2\tau^2+\frac{2}{3\beta^\frac{1}{2}}\int_M Gf^2\tau^2 \notag\\
&\leq \frac{\beta}{3}\int_M |\nabla v|^3\phi^2+\frac{C}{\beta^\frac{1}{2}}
\end{align}
where we have used Lemma 2.1 in the second inequality. Next, expanding the gradient of $\eta$ shows
\begin{equation*}
\int_M |\nabla \eta|^2G|\nabla v|\leq 2\int_M|\nabla \tau|^2f^2|\nabla G|+2\int_M\tau^2|\nabla f|^2|\nabla G|.
\end{equation*}
Since $\tau$ is assumed to be constant everywhere except on the concentric annuli $(B_p(2)\setminus B_p(1))\cup (B_p(2R)\setminus B_p(R))$, we use the Cauchy-Schwarz inequality, Lemma 2.1, and the co-area formula to write
\begin{align*}
\int_M |\nabla \tau|^2f^2|\nabla G|\ 
&\leq \int_{M\setminus B_p(1)}\left(\frac{|\nabla G|}{G^\frac{1}{2}}f\right)\left(G^\frac{1}{2}f\right)\\
& \leq \left(\int_{L(0,e^{-\frac{1}{\alpha}}A)}\frac{|\nabla G|^2}{G\ln^{2\gamma}(AG^{-1})}\right)^\frac{1}{2}\left(\int_{M\setminus B_p(1)}Gf^2\right)^\frac{1}{2}\\
& \leq C\left(\int_0^{e^{-\frac{1}{\alpha}}A}\frac{dr}{r\ln^{2\gamma}(A/r)}\right)^\frac{1}{2}\\
& \leq C.
\end{align*}
From calculations in Lemma 2.1, recall
\begin{equation*}
|\nabla f| = \gamma \frac{|\nabla G|}{G\ln^{\gamma+1}(AG^{-1})},
\end{equation*}
from which it follows 
\begin{equation*}
\tau^2|\nabla f|^2|\nabla G| = \frac{\gamma^2}{\ln^2(AG^{-1})}|\nabla v|^3\phi^2.
\end{equation*}
From the definition of $A = e^\frac{1}{\alpha}\max_{\partial B_p(1)}G$ we know that $\ln^2(AG^{-1})\geq \frac{1}{\alpha^2}$ on $M\setminus B_p(1)$, and hence
\begin{equation*}
\int_M |\nabla \eta|^2G|\nabla v|\leq 2\int_M |\nabla\tau|^2f^2|\nabla G| +2\gamma^2\alpha^2\int_M |\nabla v|^3\phi^2 .
\end{equation*}
In (\ref{eq2}), these estimates show
\begin{equation*}
\left(\frac{1}{n-1}-C\delta-\frac{C}{\delta}\left(\gamma^2\alpha^2+\beta\right)\right)\int_M|\nabla v|^3\phi^2\leq \frac{C}{\delta}\left(1+\frac{1}{\beta^\frac{1}{2}}\right)+\int_M w|\nabla v| \phi^2.
\end{equation*}

We lastly turn our attention to the curvature term. Since our hypotheses state $w\in L^s(M)$ for some $s\geq\frac{3}{2}$, an index $t$ conjugate to $s$ satisfies $1\leq t\leq 3$ and therefore
\begin{align*}
\int_M w|\nabla v|\phi^2 &\leq \frac{1}{s\delta^\frac{s}{t}}\int_M w^s\phi^2+\frac{\delta}{t}\int_M |\nabla v|^t\phi^2\\
&\leq \frac{C}{\delta^\frac{s}{t}}+\frac{\delta}{t}\int_M|\nabla v|^t\phi^2.
\end{align*}
Once more applying Young's inequality, we see
\begin{equation*}
\int_M |\nabla v|^t\phi^2\leq \left(\frac{t-1}{2}\right)\int_M |\nabla v|^3+\left(\frac{3-t}{2}\right)\int_M|\nabla v|\phi^2.
\end{equation*}
Then by (\ref{eq3}), since $\frac{t-1}{2}, \frac{3-t}{2}\leq 1$ for $1\leq t\leq 3$,
\begin{align*}
\frac{\delta}{t}\int_M |\nabla v|^t\phi^2 &\leq \delta\int_M |\nabla v|^3+\delta\left\{\beta\int_M |\nabla v|^3\phi^2+\frac{C}{\beta^\frac{1}{2}}\right\}
\end{align*}
which implies
\begin{equation*}
\int_M w|\nabla v|\phi^2\leq \frac{C}{\delta^\frac{s}{t}} + C\delta \int_M|\nabla v|^3\phi^2 + C\frac{\delta}{\beta^\frac{1}{2}}.
\end{equation*}
We may then conclude
\begin{equation*}
\left(\frac{1}{n-1}-C\delta -\frac{C}{\delta}(\gamma^2\alpha^2+\beta)\right)\int_M|\nabla v|^3\phi^2\leq \frac{C}{\delta}\left(1+\frac{1}{\beta^\frac{1}{2}}\right)+\frac{C}{\delta^\frac{s}{t}}.
\end{equation*}
The parameters $\alpha,\beta,$ and $\delta$ can be chosen such that $\frac{1}{n-1}-C\delta -\frac{C}{\delta}(\gamma^2\alpha^2+\beta)>0$, thus letting $R\to\infty$ in the definition of $\phi$ proves the result.
\end{proof}

Using Lemma 2.1 and Lemma 2.2, we are ready to prove the spectral estimate stated as Theorem 1.4 in the introduction.

\begin{proof}[Proof of Theorem 1.4] Assume $M$ is nonparabolic and let $G(x)=G(p,x)$ be the minimal positive Green's function on $M$ with pole at $p$. From the improved Kato inequality in the case of a harmonic function, we have that

\begin{equation*}
|\nabla^2G|^2\geq \frac{n}{n-1}|\nabla|\nabla G||^2.
\end{equation*}
Similarly, the Bochner formula can be written as
\begin{equation*}
\frac{1}{2}\Delta |\nabla G|^2=|\nabla^2G|^2+\text{Ric}(\nabla G,\nabla G).
\end{equation*}
Combining the above statements,
\begin{equation}
\label{eq4}
    \Delta|\nabla G|\geq \frac{1}{n-1}|\nabla|\nabla G||^2|\nabla G|^{-1}+\text{Ric}(\nabla G,\nabla G)|\nabla G|^{-1}. 
\end{equation}
To make use of (\ref{eq4}) we manipulate the Poincar\'e inequality with a particular choice of test function $\varphi\in C^\infty_c(M)$ involving $|\nabla G|$. To construct $\varphi$, first define as before the set of points

\begin{equation*}
L(a,b)=\{x\in M\,:\, a<G(x)<b\}
\end{equation*}
and for $\varepsilon>0$ small, let 
\[
\chi(x) = 
\begin{cases} 
    0               &\quad x\in L(0,\varepsilon^2)\\
    \frac{\ln G(x) - \ln(\varepsilon^2)}{-\ln(\varepsilon)}         &\quad x\in L(\varepsilon^2,\varepsilon)\\
    1               &\quad x\in L(\varepsilon,\infty)
\end{cases},
\]

\[
\psi(x) = 
\begin{cases}
    0                 & \quad x \in B_p(1) \\
    r(x) - 1          & \quad x \in B_p(2) \setminus B_p(1) \\
    1                 & \quad x \in B_p(R) \setminus B_p(2) \\
    R+1 - r(x)        & \quad x \in B_p(R+1) \setminus B_p(R) \\
    0                 & \quad x \in M \setminus B_p(R+1)
\end{cases}
\]
as in Lemma 2.1. In this way we ensure that $\varphi=\chi\psi$ vanishes near $p$ and outside of a compact set. The Poincar\'e inequality with $\varphi$ chosen as the test function reads
\begin{equation*}
\lambda_0(M)\int_M|\nabla G|\varphi^2\leq \int_M|\nabla(|\nabla G|^{\frac{1}{2}}\varphi)|^2.
\end{equation*}
Upon expanding the right side we arrive at 
\begin{align}\label{eq5}
        \lambda_0(M)\int_M|\nabla G|\varphi^2\leq &\frac{1}{4}\int_M|\nabla|\nabla G||^2|\nabla G|^{-1}\varphi^2+\int_M|\nabla G||\nabla \varphi|^2\notag\\
        &+\int_M\varphi\left<\nabla \varphi,\nabla|\nabla G|\right>.
\end{align}
The third term in the right side of the previous inequality can be estimated using the Cauchy-Schwarz inequality 
\begin{equation*}
\int_M\varphi\left<\nabla \varphi,\nabla|\nabla G|\right>\leq \int_M \varphi|\nabla \varphi||\nabla|\nabla G||.
\end{equation*}
In order to simplify, we write $\varphi|\nabla \varphi||\nabla|\nabla G||=\varphi |\nabla|\nabla G|||\nabla G|^{-\frac{1}{2}}|\nabla G|^{\frac{1}{2}}|\nabla \varphi|$ and apply Cauchy's inequality for $\delta>0$. This yields
\begin{equation*}
\int_M\varphi\left<\nabla \varphi,\nabla|\nabla G|\right>\leq \delta\int_M|\nabla|\nabla G||^2|\nabla G|^{-1}\varphi^2+\frac{C}{\delta}\int_M|\nabla G||\nabla \varphi|^2.
\end{equation*}
Inserting this into (\ref{eq5}), 
\begin{equation*}
\lambda_0(M)\int_M|\nabla G|\varphi^2\leq \left(\frac{1}{4}+\delta\right)\int_M|\nabla|\nabla G||^2|\nabla G|^{-1}\varphi^2+\frac{C}{\delta}\int_M|\nabla G||\nabla\varphi|^2.
\end{equation*}
We now estimate the first term on the right side using (\ref{eq4}). Note that by definition
\begin{equation*}
\text{Ric}(\nabla G,\nabla G)\geq -w|\nabla G|^2-(n-1)|\nabla G|^2,
\end{equation*}
so we may write 
\begin{equation*}
|\nabla|\nabla G||^2|\nabla G|^{-1}\leq (n-1)\Delta |\nabla G|+(n-1)w|\nabla G|+(n-1)^2|\nabla G|.
\end{equation*}
Consequently,
\begin{align*}
\int_M |\nabla|\nabla G||^2|\nabla G|^{-1}\varphi^2\leq &(n-1)\int_M\varphi^2\Delta|\nabla G|+(n-1)\int_M w|\nabla G|\varphi^2\\
&+(n-1)^2\int_M|\nabla G|\varphi^2.
\end{align*}
Integration by parts gives $\int_M\varphi^2\Delta|\nabla G|=-\int_M\left<\nabla \varphi^2,\nabla|\nabla G|\right>$, hence repeating previous arguments implies that
\begin{align*}
\int_M |\nabla|\nabla G||^2|\nabla G|^{-1}\varphi^2\leq& \delta\int_M|\nabla|\nabla G||^2|\nabla G|^{-1}\varphi^2+\frac{C}{\delta}\int_M|\nabla G||\nabla \varphi|^2\\
&+(n-1)\int_Mw|\nabla G|\varphi^2+(n-1)^2\int_M|\nabla G|\varphi^2.
\end{align*}
Equivalently,
\begin{align*}
\int_M|\nabla|\nabla G||^2|\nabla G|^{-1}\varphi^2\leq &\frac{(n-1)^2}{1-\delta}\int_M|\nabla G|\varphi^2+\frac{n-1}{1-\delta}\int w|\nabla G|\varphi^2\\
&+\frac{C}{\delta(1-\delta)}\int_M|\nabla G||\nabla \varphi|^2.
\end{align*}
In turn, we see
\begin{align*}
        \left(\lambda_0(M)-\frac{(n-1)^2\left(\frac{1}{4}+\delta\right)}{1-\delta}\right)\int_M|\nabla G|\varphi^2  \leq &(n-1)\left(\frac{\frac{1}{4}+\delta}{1-\delta}\right)\int_M w|\nabla G|\varphi^2\\
        &+C\left(\frac{1}{\delta}+\frac{\frac{1}{4}+\delta}{\delta(1-\delta)}\right)\int_M|\nabla G||\nabla \varphi|^2.
\end{align*}

We next prove a bound for $\int_M|\nabla G||\nabla \varphi|^2$. Using the definition of $\varphi$, 
\begin{equation*}
\int_M|\nabla G||\nabla\varphi|^2\leq 2\int_M|\nabla G||\nabla\chi|^2\psi^2+2\int_M|\nabla G||\nabla\psi|^2\chi^2.
\end{equation*}
The second term in the previous line can be estimated as follows. From H\"older's inequality and the definition of $\psi$,
\begin{equation*}
\int_M|\nabla G||\nabla\psi|^2\chi^2\leq\left(\int_{M\setminus B_p(R)}|\nabla G|^2\right)^{\frac{1}{2}}\left(\int_{(B_p(R+1)\setminus B_p(R))\cap L(\varepsilon^2,\infty)}\chi^2\right)^{\frac{1}{2}}+C.
\end{equation*}
The first factor is bounded directly by estimates in \cite{LW}, while the second is controlled by
\begin{align*}
\int_{(B_p(R+1)\setminus B_p(R))\cap L(\varepsilon^2,\infty)}\chi^2 &\leq \frac{1}{\varepsilon^4}\int_{M\setminus B_p(R)}G^2\\
&\leq \frac{C}{\varepsilon^4}e^{-2\sqrt{\lambda_0(M)}R}.
\end{align*}

The term $\int_M |\nabla G||\nabla \chi|^2\psi^2$ will be bounded using Lemma 2.1. Note once again that $\nabla \chi$ is only nonzero on $L(\varepsilon^2,\varepsilon)$, and in this region

\begin{equation*}
|\nabla \chi|^2 \leq \frac{1}{\ln^2(\varepsilon)}\frac{|\nabla G|^2}{G^2}.
\end{equation*}
As a result, applying Lemma 2.2,
\begin{align*}
\int_M|\nabla G||\nabla \chi|^2\psi^2 &\leq \frac{C}{\ln^2(\varepsilon)}\int_{L(\varepsilon^2,\varepsilon)} \frac{|\nabla G|^3}{G^2}\\
& \leq C\ln^{2\gamma - 2}(1+\varepsilon^{-1}).
\end{align*}
Choosing $1<2\gamma<2$, the quantity on the right side goes to 0 as $\varepsilon\to 0$. The estimate obtained is
\begin{equation*}
\int_M|\nabla G||\nabla\varphi|^2\leq C\ln^{2\gamma -2}(1+\varepsilon^{-1})+\frac{C}{\varepsilon^2}e^{-2\sqrt{\lambda_0(M)}R}+C.
\end{equation*}
Plugging this estimate into (\ref{eq5}) we have
\begin{align}
\left(\lambda_0(M)-\frac{\frac{1}{4}+\delta}{1-\delta}(n-1)^2\right)\int_M|\nabla G|\varphi^2 \leq &\frac{C}{\delta}\left\{\ln^{2\gamma-2}(1+\varepsilon^{-1})+\frac{e^{-2\sqrt{\lambda_0(M)}R}}{\varepsilon^2}+C\right\} \notag\\
& +\frac{(n-1)\left(\frac{1}{4}+\delta \right)}{1-\delta}\int_M w|\nabla G|\varphi^2.\notag
\end{align}

Having dealt with the term $\int_M|\nabla G||\nabla \varphi|^2$ in (\ref{eq5}), we turn our attention to $\int_Mw|\nabla G|\varphi^2$. Under the assumption that $w\in L^s(M)$, $s\geq\frac{3}{2}$, we may apply Young's inequality and find for $\delta>0$
\begin{align*}
\int_Mw|\nabla G|\varphi^2 &\leq \frac{1}{s\delta^\frac{s}{t}}\int_M w^s\varphi^2 + \frac{\delta}{t}\int_M |\nabla G|^t\varphi^2\\
&\leq \frac{C}{\delta^\frac{s}{t}}+\delta \int_M|\nabla G|^t\varphi^2
\end{align*}
where $1\leq t\leq 3$ is H\"older conjugate to $s$. By the same reasoning that was used in Lemma 2.2, we know
\begin{equation*}
\int_M |\nabla G|^t\varphi^2\leq \left(\frac{t-1}{2}\right)\int_M |\nabla G|^3\varphi^2+\left(\frac{3-t}{2}\right)\int_M|\nabla G|\varphi^2.
\end{equation*}
Since $G$ is bounded outside of $B_p(1)$, Lemma 2.2 implies $\int_M |\nabla G|^3\varphi^2<\infty$ and hence
\begin{equation*}
\int_M w|\nabla G|\varphi^2\leq \frac{C}{\delta^\frac{s}{t}}+C\delta + \delta\int_M |\nabla G|\varphi^2.
\end{equation*}
From this estimate we have
\begin{align*}
&\left(\lambda_0(M)-\frac{(n-1)\left(\frac{1}{4}+\delta\right)(n-1+\delta)}{1-\delta}\right)\int_M |\nabla G|\varphi^2\\
&\leq C(\delta)\left\{\ln^{2\gamma-2}(1+\varepsilon^{-1})+\frac{1}{\varepsilon^2}e^{-2\sqrt{\lambda_0(M)}R}+C\right\}+\frac{C}{\delta^\frac{s}{t}} + C\left(\frac{\delta(n-1)\left(\frac{1}{4}+\delta\right)}{1-\delta}\right).
\end{align*}
We let $R\to\infty$ and then $\varepsilon\to 0$ to find that for any $\delta>0$
\begin{align}\label{eq6}
\left(\lambda_0(M)-\frac{(n-1)\left(\frac{1}{4}+\delta\right)(n-1+\delta)}{1-\delta}\right)\int_{M\setminus B_p(2)} |\nabla G| \leq C(\delta).
\end{align}
However, note that 
\begin{equation*}
\int_{B_p(r)}\frac{\partial G}{\partial r}=-1,
\end{equation*}
hence
\begin{equation*}
\int_{M\setminus B_p(2)}|\nabla G|\geq -\int_2^\infty dr\int_{\partial B(p,r)}\frac{\partial G}{\partial r}=\infty.
\end{equation*}
Therefore (\ref{eq6}) can only hold if
\begin{equation*}
\lambda_0(M)\leq \frac{(n-1)\left(\frac{1}{4}+\delta\right)(n-1+\delta)}{1-\delta}.
\end{equation*}
By taking the limit as $\delta\to 0$, we conclude
\begin{equation*}
\lambda_0(M)\leq \frac{(n-1)^2}{4}.
\end{equation*}
\end{proof}

\section{COUNTEREXAMPLE FOR $s<\frac{1}{2}$}

It is natural to ask whether the condition $s\geq \frac{3}{2}$ is necessary in the statement of Theorem 1.4. Though we do not show this lower bound is sharp, we construct a manifold $M$ such that $\int_M w^s<\infty$, but $\lambda_0(M)>\frac{(n-1)^2}{4}$ where $s< \frac{1}{2}$ and $w$ is the curvature quantity defined in section 1. In particular we prove the following result.
\begin{theorem}
    There exists a Riemannian $n$-manifold $(M,g)$ with $\int_M w^s <\infty$ for $s<\frac{1}{2}$ such that $\lambda_0(M)>\frac{(n-1)^2}{4}$.
\end{theorem}
Note that if $M$ supports a smooth test function $\phi>0$ and there exists a constant $\varepsilon>0$ for which
\begin{equation*}
\Delta \phi\leq  -\frac{(n-1+\varepsilon)^2}{4}\phi,
\end{equation*}
the lower bound $\lambda_0(M)>\frac{(n-1)^2}{4}$ follows. This is the premise on which our proof is based. More specifically, in the following construction we utilize properties of harmonic functions. Consider $\phi = H^{\frac{1}{2}}$ where $H$ is harmonic; if we can show there exists $\varepsilon>0$ such that $|\nabla H|\geq (n-1+\varepsilon)H$, then
\begin{align*}
\Delta \phi &=-\frac{1}{4}\frac{|\nabla H|^2}{H^2}\phi\\
&\leq -\frac{(n-1+\varepsilon)^2}{4}\phi.
\end{align*}
Thus to prove Theorem 3.1 it suffices to show there exists a manifold $M$ and a harmonic function $H$ on $M$ satisfying 
\begin{equation}\label{greensfuncineq}
(n-1+\varepsilon)H\leq |\nabla H|.
\end{equation} 

We construct a manifold which supports a harmonic function satisfying (\ref{greensfuncineq}). Let $M$ be the product manifold $M=\mathbb{R}\times N$, $\text{dim}(N) = n-1$, with warped metric $ds_M^2=dt^2+e^{2h(t)}ds_N^2$. Denote the area of $\{x\}\times N\subset \mathbb{R}\times N$ by $A(x)$. Under the normalizing assumption that the manifold $N$ is of unit volume, define the function
\begin{equation*}
H(t)=\int_{t}^\infty \frac{dx}{A(x)}=\int_t^\infty e^{-(n-1)h(x)}dx.
\end{equation*}
We see because of the simple form of the metric that $\Delta t= (n-1)h'(t)$ and $\Delta H=0$. The inequality (\ref{greensfuncineq}) that we wish to prove then takes on the particular form
\begin{equation}\label{warpedproductineq}
    (n-1+\varepsilon)\int_{R}^\infty e^{-(n-1)h(t)}dt\leq e^{-(n-1)h(R)}
\end{equation}
for each $R$. In the setting of warped products there is also a useful expression for the components $R_{ij}$ of the Ricci curvature of $M$ in terms of the warping function $h$ and the Ricci curvature $\widetilde{R}_{ij}$ of $N$:
\[
R_{ij}=
\begin{cases}
    -(n-1)(h''(t)+(h'(t))^2)&\quad i=j=1\\
    e^{-2h(t)}\widetilde{R}_{ij}-(h''(t)+(n-1)(h'(t))^2)\delta_{ij}&\quad i,j>1\\
    0&\quad \text{otherwise}
\end{cases}.
\]
If we assume the Ricci curvature of $N$ is nonnegative, it follows that 
\begin{align*}
w&=\left\{\inf_{|v|=1}\{\text{Ric}_M(v,v)+(n-1)g_M(v,v)\}\right\}_{-}\\
&\leq \max\{(n-1)(h''+(h')^2-1), h''+(n-1)((h')^2-1)\}.
\end{align*}
This shows that in a warped product setting, the extent to which the Ricci curvature falls below $-(n-1)$ is governed solely by the function $h$. We need only to construct this real-valued function such that $w\in L^s(M)$ for $s<\frac{1}{2}$. 

Let $\{\delta_k\}_{k\in\mathbb{N}\cup\{0\}}$ be a decreasing sequence of positive real numbers $\delta_k<1$ such that $\lim_{k\to\infty}\delta_k=0$ and fix $\sigma>0$, $\Lambda\geq 0$. On intervals $I_k = [\Lambda+k,\Lambda + k + \delta_k]$, $J_k = [\Lambda+k+\delta_k,\Lambda+k+1]$, define
\[
h(t) = 
\begin{cases}
    h_{1,k}(t) &\quad t\in I_k\\
    h_{2,k}(t)=t+k\sigma &\quad t\in J_k
\end{cases}
\]
where the functions $h_{1,k}$ are determined by interpolating between the linear functions $h_{2,k-1}$ and $h_{2,k}$. That is,
\begin{equation*}
h_{1,0} = \chi_0(t-\sigma)+(1-\chi_0)h_{2,0},\quad h_{1,k}=\chi_k h_{2,k-1}+(1-\chi_k)h_{2,k} \;\;\text{for $k\geq 1$.}
\end{equation*}
We take $\chi_k:\mathbb{R}\to\mathbb{R}$ to be a smooth step function with $\chi_k(t)= 1$ for $t\leq \Lambda+k$, $\chi_k(t)= 0$ for all $t\geq \Lambda+k+\delta_k$, $\chi_k^{(l)}(\Lambda+k)=\chi_k^{(l)}(\Lambda+k+\delta_k)=0$ for all $l\geq 1$, and
\begin{equation}\label{chiestimate}
|\chi_k^{(l)}(t)|\leq C_l\delta_k^{-l},\; \text{for all $t\in I_k$\, and\, $l\geq 1$}.
\end{equation}
We see that because $h_{2,k}$ is linear with slope 1 for each $k$, $w\big|_{J_k}= 0$ as both $(n-1)(h''+(h')^2-1)$ and $h''+(n-1)((h')^2-1)$ are identically zero. As a result,
\begin{equation*}
\int_M w^s=\sum_{k=0}^\infty \int_{I_k}w^se^{(n-1)h_{1,k}(t)}dt.
\end{equation*}
Using the fact that $h_{2,k}(t)-h_{2,k-1}(t)=\sigma$ for all $t$,
\begin{align*}
&h_{1,k}'(t)=-\sigma\chi_k'(t)+1,\\
&h_{1,k}''(t)=-\sigma\chi_k''(t).
\end{align*}
In turn, using (\ref{chiestimate}) and the definition of $w$ we find that on the intervals $I_k$
\begin{equation*}
w\leq C\delta_k^{-2}
\end{equation*}
where $C$ depends on $\sigma$ and the dimension $n$. Estimating $h_{1,k}(t)\leq h_{2,k}(t)=t+k\sigma$ gives
\begin{equation*}
\sum_{k=0}^\infty \int_{I_k}w^se^{(n-1)h_{1,k}(t)}dt\leq C\sum_{k=0}^\infty \delta_k^{-2s}\int_{\Lambda+k}^{\Lambda+k+\delta_k}e^{(n-1)(t+k\sigma)}dt.
\end{equation*}
By estimating $e^{(n-1)(t+k\sigma)}\leq e^{(n-1)(\Lambda+k+\delta_k+k\sigma)}$ over each interval $I_k$, this implies that
\begin{align*}
\int_Mw^s &\leq C\sum_{k=0}^\infty \delta_k^{-2s}e^{(n-1)(\Lambda+k+\delta_k+k\sigma)}\\
&=C(\Lambda)\sum_{k=0}^\infty \delta_k^{-2s}e^{(n-1)(1+\sigma)k}.
\end{align*}
It is then possible to choose $\delta_k=e^{-\frac{1}{1-2s}((n-1)(1+\sigma)+1)k},$ and in this case
\begin{equation*}
\sum_{k=0}^\infty \delta_k^{1-2s}e^{(n-1)(1+\sigma)k}=\sum_{k=0}^\infty e^{-k}<\infty
\end{equation*}
exactly when $s<\frac{1}{2}$.

We next must prove that (\ref{warpedproductineq}) is satisfied in the setting defined above. This means we want to show there exists $\varepsilon>0$ such that for every $R$
\begin{equation*}
(n-1+\varepsilon)\int_R^\infty e^{-(n-1)h(t)}dt<e^{-(n-1)h(R)}.
\end{equation*}
There are two cases depending on if $R\in I_k$ or $R\in J_k$. First suppose $R\in J_k$, i.e. $\Lambda +k+\delta_k\leq R\leq \Lambda +k+1$. We can write
\begin{align*}
\int_R^\infty e^{-(n-1)h(t)}dt=&\int_R^{\Lambda+k+1}e^{-(n-1)(t+k\sigma)}dt+\int_{I_{k+1}}e^{-(n-1)h_{1,k+1}(t)}dt\\
&+\int_{\Lambda+k+1+\delta_{k+1}}^\infty e^{-(n-1)h(t)}dt
\end{align*}
and estimate each integral on the right side. The first can be computed explicitly
\begin{equation*}
\int_R^{\Lambda+k+1}e^{-(n-1)(t+k\sigma)}dt=\frac{1}{n-1}\left(e^{-(n-1)h(R)}-e^{-(n-1)h(\Lambda+k+1)}\right).
\end{equation*}
The second integral is estimated by using the fact that $I_{k+1}$ is an interval of length $\delta_{k+1}$ and $e^{-(n-1)h_{1,k+1}(t)}$ is bounded above by $e^{-(n-1)h(\Lambda+k+1)}$. This gives
\begin{equation*}
\int_{I_{k+1}}e^{-(n-1)h_{1,k+1}(t)}dt\leq \delta_{k+1}e^{-(n-1)h(\Lambda+k+1)}.
\end{equation*}
Lastly, for $t\geq \Lambda+k+1+\delta_k$ we know $h(t)\geq t+(k+1)\sigma$, so
\begin{equation*}
\int_{\Lambda+k+1+\delta_{k+1}}^\infty e^{-(n-1)h(t)}dt\leq \frac{1}{n-1}e^{-(n-1)h(\Lambda+k+1+\delta_{k+1})}.
\end{equation*}
Combining these estimates we have
\begin{equation*}
\int_R^\infty e^{-(n-1)h(t)}dt\leq \frac{C(k, R)}{n-1}e^{-(n-1)h(R)}
\end{equation*}
where
\begin{align*}
C(k,R)=&1-e^{-(n-1)(h(\Lambda+k+1)-h(R))}+(n-1)\delta_{k+1}e^{-(n-1)(h(\Lambda+k+1)-h(R))}\\
&+e^{-(n-1)(h(\Lambda+k+1+\delta_{k+1})-h(R))}.
\end{align*}
From the previous inequality, we must show that there exists $\varepsilon>0$ independent of $k$ and $R$ such that 
\begin{equation*}
\frac{n-1+\varepsilon}{n-1}C(k,R)e^{-(n-1)h(R)}<e^{-(n-1)h(R)}.
\end{equation*}
It is then enough to prove there exists a universal constant $C$ such that $C(k,R)\leq C<1$, so it remains to estimate $C(k,R)$. Note from the increasing nature of $h$ and the permissible range of $R$ that
\begin{align*}
e^{-(n-1)(h(\Lambda+k+1)-h(R))}&\geq e^{-(n-1)(h(\Lambda+k+1)-h(\Lambda+k+\delta_k))}\\
&=e^{-(n-1)(1-\delta_k)}.
\end{align*}
Similarly,
\begin{equation*}
(n-1)\delta_{k+1}e^{-(n-1)(h(\Lambda+k+1)-h(R))}\leq (n-1)\delta_{k+1},
\end{equation*}
and also
\begin{align*}
e^{-(n-1)(h(\Lambda+k+1+\delta_{k+1})-h(R))}&\leq e^{-(n-1)(h(\Lambda+k+1+\delta_{k+1})-h(\Lambda+k+1))}\\
&=e^{-(n-1)(\delta_{k+1}+\sigma)}.
\end{align*}
Finally, because $\delta_k$ is decreasing in $k$ we find
\begin{equation*}
C(k,R)\leq 1-e^{-(n-1)}+(n-1)\delta_1+e^{-(n-1)\sigma}.
\end{equation*}
By taking $\delta_1$ sufficiently small and $\sigma$ sufficiently large, 
\begin{align*}
0&<1-e^{-(n-1)}+(n-1)\delta_1+e^{-(n-1)\sigma}\\
&<1.
\end{align*}
Thus we may set $C=1-e^{-(n-1)}+(n-1)\delta_1+e^{-(n-1)\sigma}$. This proves (\ref{warpedproductineq}) for $R\in J_k$.

For $R\in I_k$, i.e. $\Lambda + k \leq R \leq \Lambda + k + \delta_k$, we estimate the left side of (\ref{warpedproductineq}) following a similar process to the previous case. By writing
\begin{align}\label{eq10}
\int_R^\infty e^{-(n-1)h(t)}dt = &\int_R^{\Lambda+k+\delta_k}e^{-(n-1)h(t)}dt+\int_{J_k}e^{-(n-1)h(t)}dt\notag \\
&+\int_{I_{k+1}}e^{-(n-1)h(t)}dt+\int_{\Lambda+k+1+\delta_{k+1}}^\infty e^{-(n-1)h(t)}dt
\end{align}
we can bound $h(t)$ within specific regions of the domain. First, since $h(t)$ is an increasing function and $R$ is within $I_k$ we know 
\begin{equation*}
\int_R^{\Lambda+k+\delta_k}e^{-(n-1)h(t)}dt\leq \delta_ke^{-(n-1)h(R)}.
\end{equation*}
Next, as $h(t) = h_{2,k}(t) = t+k\sigma$ on $J_k$, we can evaluate explicitly
\begin{align*}
\int_{J_k}e^{-(n-1)h(t)}dt&=\frac{1}{n-1}\left(e^{-(n-1)(\Lambda +k+\delta_k+k\sigma)}-e^{-(n-1)(\Lambda + k+1+k\sigma  )}\right)\\
&=\frac{e^{-(n-1)(\Lambda+k+\delta_k+k\sigma)}}{n-1}\left(1-e^{-(n-1)(1-\delta_k)}\right)\\
&\leq \frac{e^{-(n-1)h(R)}}{n-1}\left(1-e^{-(n-1)(1-\delta_k)}\right).
\end{align*}
Once again given that the width of $I_{k+1}$ is $\delta_{k+1}$ and $h(\Lambda+k+1)\geq h(R)$,
\begin{equation*}
\int_{I_{k+1}}e^{-(n-1)h(t)}dt\leq \delta_{k+1}e^{-(n-1)h(R)}.
\end{equation*}
Lastly, since $h(t)\geq t+(k+1)\sigma = h_{2,k+1}(t)$ for all $t\geq \Lambda+k+1+\delta_{k+1}$ we have
\begin{align*}
\int_{\Lambda+k+1+\delta_{k+1}}^\infty e^{-(n-1)h(t)}dt &\leq \frac{1}{n-1}e^{-(n-1)(\Lambda+k+1+\delta_{k+1}+(k+1)\sigma)}\\
&\leq \frac{1}{n-1}e^{-(n-1)h(R)}e^{-(n-1)(1+\delta_{k+1}-\delta_k+\sigma)}
\end{align*}
where we have used the fact that $h(\Lambda+k+1+\delta_{k+1}) - h(R)\geq h(\Lambda+k+1+\delta_{k+1}) - h(\Lambda+k+\delta_k)$. Combining estimates in (\ref{eq10}) shows
\begin{equation*}
\int_R^\infty e^{-(n-1)h(t)}dt\leq \frac{D(k)}{n-1}e^{-(n-1)h(R)}
\end{equation*}
with the constant $D(k)$ given by
\begin{equation*}
D(k) = (n-1)\delta_k+(n-1)\delta_{k+1}+1-e^{-(n-1)(1-\delta_k)}+e^{-(n-1)(1+\delta_{k+1}-\delta_k+\sigma)}.
\end{equation*}
By the same reasoning as in the previous case, we verify there exists a constant $D$ independent of $k$ such that $D(k)\leq D<1$. From the definition of the sequence $\{\delta_k\}$, we see immediately
\begin{equation*}
D(k)\leq 2(n-1)\delta_0+1-e^{-(n-1)}+e^{-(n-1)(1-\delta_0+\sigma)}.
\end{equation*}
Provided $\delta_0$ is sufficiently small and $\sigma$ is sufficiently large, we may take 
\begin{align*}
D&=2(n-1)\delta_0+1-e^{-(n-1)}+e^{-(n-1)(1-\delta_0+\sigma)}\\
&<1.
\end{align*}
This completes the proof of (\ref{warpedproductineq}), and hence of Theorem 3.1.

\specialsection*{ACKNOWLEDGMENT}
\quad This work has been done as part of my PhD at the University of Connecticut. I would like to thank my advisor Ovidiu Munteanu for many beneficial discussions, and Jiaping Wang for helpful suggestions along the way.
\bibliographystyle{amsplain}
\bibliography{Cite}

\end{document}